\documentclass[a4paper]{amsart}

\usepackage[utf8]{inputenc}
\usepackage[english]{babel}

\usepackage[T1]{fontenc}
\usepackage{lmodern}  

\usepackage{my-math-symbol}
\usepackage{my-cleveref}
\usepackage{my-equation-numbering}
\usepackage{comment}
\usepackage{tikz}
\usetikzlibrary{cd}
\usepackage[colorinlistoftodos]{todonotes}

\newcommand{\GJMS}{P}
\newcommand{\dilation}{\delta}
\newcommand{\ltrans}{l}
\newcommand{\inversion}{\iota}

\newcommand{\Hsymbol}{S_{H}}
\newcommand{\Hpsido}{\Psi_{H}}
\newcommand{\HSobolev}{W_{H}}

\usepackage[non-sorted-cites,initials,alphabetic,nobysame]{amsrefs}

\AtBeginDocument{%
	\def\MR#1{}
}

\title{Lower bounds for the critical CR GJMS operator}
\author{Yuya Takeuchi}
\address{Division of Mathematics \\ Institute of Pure and Applied Sciences \\ University of Tsukuba
	\\ 1- 1- 1 Tennodai, Tsukuba, Ibaraki 305-8571 Japan}
\email{ytakeuchi@math.tsukuba.ac.jp, yuya.takeuchi.math@gmail.com}



\thanks{This work was supported by JSPS KAKENHI Grant Number JP25K17247.}

\begin{document}

\begin{abstract}
	We establish lower bounds for the critical CR GJMS operator
	on closed pseudo-Hermitian manifolds.
	In dimension $2 n + 1$,
	assuming embeddability,
	we obtain a lower bound involving the Heisenberg Sobolev norm of order $n + 1$,
	modulo an $L^{2}$ term,
	on the orthogonal complement of the kernel.
	In dimension three,
	we prove that the CR Paneitz operator is bounded below without assuming embeddability.
	The proofs combine the Heisenberg calculus
	with a square-root construction based on the joint functional calculus
	of the sub-Laplacian and the Reeb vector field on the Heisenberg group.
	As an application, we show that the infinitely many negative eigenvalues
	of the CR Paneitz operator on the Rossi sphere have zero as their only accumulation point.
	Consequently, this operator does not have closed range.
\end{abstract}

\maketitle

\section{Introduction}
\label{section:introduction}

Invariant differential operators play a central role in conformal and CR geometry.
Their analytic properties are closely related to geometric problems
such as the Yamabe problem and the constant $Q$-curvature problem.

The GJMS operators,
introduced by Graham, Jenne, Mason, and Sparling~\cite{Graham-Jenne-Mason-Sparling1992},
are conformally covariant differential operators
whose leading terms are powers of the Laplacian.
On a Riemannian manifold of dimension $d \geq 3$,
their construction gives operators of order $2 k$
for every positive integer $k$ when $d$ is odd,
and for $1 \leq k \leq d / 2$ when $d$ is even.
These operators are formally self-adjoint~\cite{Graham-Zworski2003}.
With the convention that the Laplacian is non-negative,
their principal symbols are positive away from the zero section.
Standard elliptic theory therefore implies that,
on a closed Riemannian manifold,
their self-adjoint realizations are bounded below and have compact resolvent.
These properties remain valid at the critical order $2 k = d$ in even dimensions.

Gover and Graham~\cite{Gover-Graham2005} constructed CR analogues of these operators
using the Fefferman metric.
On a strictly pseudoconvex CR manifold of dimension $2 n + 1$,
a choice of contact form $\theta$ determines formally self-adjoint CR GJMS operators $P_{k}$,
$1 \leq k \leq n + 1$, of Heisenberg order $2 k$.
For the subcritical orders $1 \leq k \leq n$,
these operators have invertible Heisenberg principal symbols and are subelliptic%
~\cite{Ponge2008-Book}*{Proposition 3.5.7}.
On closed manifolds,
their self-adjoint realizations have compact resolvent and are bounded below;
the latter property is established in \cite{Ponge2008-Book}*{Proposition 5.4.12}.
Thus both the conformal and subcritical CR GJMS operators have finite-dimensional kernels
and only finitely many negative eigenvalues,
counted with multiplicity.

The \emph{critical CR GJMS operator} $\GJMS \coloneqq P_{n + 1}$ is analytically different.
Its Heisenberg principal symbol is no longer invertible,
and it annihilates CR pluriharmonic functions.
On a closed embeddable CR manifold,
these functions already form an infinite-dimensional space,
so the kernel of $\GJMS$ is infinite-dimensional;
see \cite{Takeuchi2023-GJMS}*{Section 2}.
The lower-bound arguments for the conformal and subcritical CR cases
therefore do not apply directly to $\GJMS$,
and its spectral analysis must take this degeneracy into account.

The purpose of this paper is to establish lower bounds for the critical CR GJMS operator,
including cases in which non-negativity fails.
In dimension three,
this operator is also known as the \emph{CR Paneitz operator}.
The sign of $\GJMS$ is sensitive to the underlying CR geometry.
On a closed embeddable three-dimensional CR manifold,
the CR Paneitz operator is non-negative \cite{Takeuchi2020-Paneitz}.
In higher dimensions,
however,
the critical CR GJMS operator can have negative eigenvalues
(\cref{rem:negative-eigenvalue-in-higher-dimension}).
Negative eigenvalues also occur in dimension three
when embeddability is not assumed.
In particular,
the CR Paneitz operator on the Rossi sphere has infinitely many negative eigenvalues~\cite{Takeuchi2024-preprint}.
These examples motivate the question of whether the spectrum is nevertheless bounded below.

We use the volume form $\theta \wedge (d \theta)^{n}$ to define $L^{2}(M)$
and regard the critical CR GJMS operator $\GJMS$ as an unbounded operator on $L^{2}(M)$ with domain
\begin{equation}
	\Dom \GJMS
	= \Set{u \in L^{2}(M) | \GJMS u \in L^{2}(M)}.
\end{equation}
Here $\GJMS u$ is understood in the distributional sense.
We write $\iproduct{\cdot}{\cdot}_{0}$ and $\norm{\cdot}_{0}$
for the $L^{2}$ inner product and norm.
For $s \in \bbR$,
let $\HSobolev^{s}(M)$ be the Heisenberg Sobolev space with norm
\begin{equation}
	\norm{u}_{s}
	\coloneqq \norm{(\Delta_{b} + 1)^{s / 2} u}_{0},
\end{equation}
where $\Delta_b$ is the sub-Laplacian;
see \cref{section:Heisenberg-calculus-and-lower-bound-criterion}.

In the embeddable case,
the author proved that $\GJMS$ is self-adjoint and has closed range~\cite{Takeuchi2023-GJMS}.
Our first main result provides a lower bound that controls the Heisenberg Sobolev norm
on the orthogonal complement of the kernel,
up to an $L^{2}$ term.

\begin{theorem}
\label{thm:Garding-type-inequality-for-critical-CR-GJMS-operator}
	Let $(M, T^{1, 0} M, \theta)$ be a closed pseudo-Hermitian manifold of dimension $2 n + 1$.
	Assume that $(M, T^{1, 0} M)$ is embeddable.
	Then there exist constants $\varepsilon, C > 0$
	such that
	\begin{equation}
		\iproduct{\GJMS u}{u}_{0}
		\geq \varepsilon \norm{u}_{n + 1}^{2} - C \norm{u}_{0}^{2}
	\end{equation}
	holds for every $u \in \Dom \GJMS \cap (\Ker \GJMS)^{\perp}$.
	In particular,
	the spectrum of $\GJMS$ is bounded below.
\end{theorem}

The estimate also implies that the negative spectral subspace of $\GJMS$ is finite-dimensional;
see \cref{cor:finite-negative-eigenvalues}.
The embeddability assumption is automatic when $n \geq 2$~\cite{Boutet_de_Monvel1975}.
Thus this theorem applies to all closed pseudo-Hermitian manifolds of dimension at least five.

In dimension three,
we can dispense with embeddability for the conclusion of boundedness from below.
The self-adjointness of $P$ without an embeddability assumption
was established in \cite{Takeuchi2024-preprint}.
Our second main result is the following,
which gives an affirmative answer to \cite{Takeuchi2025}*{Problem 4.3}.

\begin{theorem}
\label{thm:bounded-below-in-dimension-three}
	Let $(M, T^{1, 0} M, \theta)$ be a closed pseudo-Hermitian manifold of dimension $3$.
	Then the spectrum of the CR Paneitz operator $\GJMS$ is bounded below.
\end{theorem}

Together with automatic embeddability in higher dimensions,
these two theorems show that the critical CR GJMS operator is bounded below
on every closed pseudo-Hermitian manifold of dimension at least three.

As an application,
we determine the accumulation behavior of the negative eigenvalues on the Rossi sphere.
The results of \cite{Takeuchi2024-preprint} imply that
there are infinitely many distinct negative eigenvalues,
each of finite multiplicity,
and that they have no accumulation point away from zero.
\cref{thm:bounded-below-in-dimension-three} rules out their escape to $- \infty$.
It follows that zero is their only accumulation point and,
consequently,
that the CR Paneitz operator on the Rossi sphere does not have closed range
(\cref{thm:negative-eigenvalues-on-Rossi-sphere}).
Thus the Rossi sphere provides a negative answer to \cite{Takeuchi2025}*{Problem 6.5}.
The same conclusions hold for the non-embeddable three-dimensional tori
with infinitely many negative eigenvalues constructed in \cite{Ho-Takeuchi2026-preprint};
see \cref{rem:infinitely-many-negative-eigenvalues}.

The Rossi sphere also shows that the embeddability assumption
in \cref{thm:Garding-type-inequality-for-critical-CR-GJMS-operator} cannot be omitted in general.
Indeed,
if the Sobolev lower bound of \cref{thm:Garding-type-inequality-for-critical-CR-GJMS-operator}
held on the Rossi sphere,
the argument of \cref{cor:finite-negative-eigenvalues} would imply that
its negative spectral subspace is finite-dimensional,
contradicting the existence of infinitely many negative eigenvalues.

The proofs are based on the Heisenberg calculus
and a square-root construction for a modified critical CR GJMS operator on the Heisenberg group.
Let $\GJMS^{0}$ and $\Delta_{b}^{0}$ denote the critical CR GJMS operator
and the sub-Laplacian on the Heisenberg group,
and let $\Pi^{0}$ be the sum of the \Szego projection and its complex conjugate.
We consider
\begin{equation}
	\wtP^{0} \coloneqq \GJMS^{0} + \Pi^{0} (\Delta_{b}^{0})^{n + 1} \Pi^{0}.
\end{equation}
Using the joint functional calculus of $\Delta_{b}^{0}$ and $- \sqrt{- 1} T^{0}$,
where $T^{0}$ is the Reeb vector field,
we construct a $U(n)$-invariant real-valued Heisenberg symbol $q^{0}$ of order $n + 1$ such that
\begin{equation}
	\sigma_{2 n + 2}^{0}(\wtP^{0}) = q^{0} \ast^{0} q^{0},
\end{equation}
where $\ast^{0}$ is the product of model symbols
induced by composition of the corresponding convolution operators.

In the embeddable case, let $\Pi$ be the orthogonal projection onto $\Ker \GJMS$ and set
\begin{equation}
	\wtP \coloneqq \GJMS + \Pi (\Delta_{b} + 1)^{n + 1} \Pi.
\end{equation}
Using the description of the partial inverse and $\Pi$ in \cite{Takeuchi2023-GJMS},
we show that $\wtP$ has a parametrix in the Heisenberg calculus.
The $U(n)$-invariance of $q^{0}$ allows the model square root
to define a global real-valued square root of the Heisenberg principal symbol of $\wtP$
with respect to the symbol product.
Quantizing this symbol,
we write $\wtP = Q^{2} + R$,
where $Q$ is formally self-adjoint of Heisenberg order $n + 1$
and $R$ has Heisenberg order at most $2 n + 1$.
A parametrix estimate for $Q$,
together with interpolation to control $R$,
gives the desired lower bound.
Since $\wtP$ agrees with $\GJMS$ on $(\Ker \GJMS)^{\perp}$,
this proves \cref{thm:Garding-type-inequality-for-critical-CR-GJMS-operator}.

For \cref{thm:bounded-below-in-dimension-three},
we replace $\Pi$ by the spectral projection $\pi_{\lambda} = E(\clcl{- \lambda}{\lambda})$,
where $\lambda > 0$ and $E$ is the spectral resolution of $\GJMS$.
The Heisenberg pseudodifferential description of this projection
and the corresponding partial inverse,
established in \cite{Takeuchi2024-preprint},
allows us to apply the same square-root argument to
$\GJMS + \pi_{\lambda} (\Delta_{b} + 1)^{2} \pi_{\lambda}$.
This proves boundedness from below without an embeddability assumption.

This paper is organized as follows.
\cref{section:CR-manifolds} recalls the basic facts on CR geometry.
In \cref{section:model-operators-on-the-Heisenberg-group},
we review model operators on the Heisenberg group,
and in \cref{section:square-root-on-Heisenberg-group}
we construct the square root of the modified model operator.
\cref{section:Heisenberg-calculus-and-lower-bound-criterion} recalls the Heisenberg calculus
and establishes the lower-bound criterion used in the proofs.
We prove \cref{thm:Garding-type-inequality-for-critical-CR-GJMS-operator}
in \cref{section:lower-bounds-in-the-embeddable}.
\cref{section:CR-Paneitz-operator} is devoted to \cref{thm:bounded-below-in-dimension-three}
and applications to the Rossi sphere and non-embeddable three-dimensional tori.

\section{CR manifolds}
\label{section:CR-manifolds}

Let $M$ be an orientable smooth $(2 n + 1)$-dimensional manifold without boundary.
A \emph{CR structure} is a rank $n$ complex subbundle $T^{1, 0} M$
of the complexified tangent bundle $T M \otimes \bbC$ such that
\begin{equation}
	T^{1, 0} M \cap T^{0, 1} M = 0, \qquad
	\comm{\Gamma(T^{1 ,0} M)}{\Gamma(T^{1, 0} M)} \subset \Gamma(T^{1, 0} M),
\end{equation}
where $T^{0, 1} M$ is the complex conjugate of $T^{1, 0} M$ in $T M \otimes \bbC$.
Define a hyperplane bundle $H M$ of $T M$ by $H M \coloneqq \Re T^{1, 0} M$.
A typical example of a CR manifold is a real hypersurface $M$
in an $(n + 1)$-dimensional complex manifold $X$;
this $M$ has the canonical CR structure
\begin{equation}
	T^{1, 0} M
	\coloneqq T^{1, 0} X |_{M} \cap (T M \otimes \bbC).
\end{equation}

Take a nowhere-vanishing real one-form $\theta$ on $M$
such that $\theta$ annihilates $T^{1, 0} M$.
The \emph{Levi form} $\calL_{\theta}$ with respect to $\theta$ is the Hermitian form
on $T^{1,0} M$ defined by
\begin{equation}
	\calL_{\theta}(Z, W)
	\coloneqq - \sqrt{- 1} \, d \theta(Z, \ovW), \qquad Z, W \in T^{1, 0} M.
\end{equation}
A CR structure $T^{1, 0} M$ is said to be \emph{strictly pseudoconvex}
if the Levi form is positive definite for some $\theta$;
such a $\theta$ is called a \emph{contact form}.
The triple $(M, T^{1, 0} M, \theta)$ is called a \emph{pseudo-Hermitian manifold}.
Denote by $T$ the \emph{Reeb vector field} with respect to $\theta$; 
that is,
the unique vector field satisfying
\begin{equation}
	\theta(T) = 1, \qquad T \contr d\theta = 0.
\end{equation}

Define an operator $\delb_{b} \colon C^{\infty}(M) \to \Gamma((T^{0, 1} M)^{\ast})$ by
\begin{equation}
	\delb_{b} f
	\coloneqq d f|_{T^{0, 1} M}.
\end{equation}
A smooth function $f$ is called a \emph{CR holomorphic function}
if $\delb_{b} f = 0$.
A \emph{CR pluriharmonic function} is a real-valued smooth function
that is locally the real part of a CR holomorphic function.

The Levi form induces a Hermitian metric on $(T^{0, 1} M)^{\ast}$.
By using this Hermitian metric and the volume form $\theta \wedge (d \theta)^{n}$,
we obtain the formal adjoint
\begin{equation}
	\delb_{b}^{\ast} \colon \Gamma((T^{0, 1} M)^{\ast}) \to C^{\infty}(M)
\end{equation}
of $\delb_{b}$.
The \emph{Kohn Laplacian} $\Box_{b}$ and the \emph{sub-Laplacian} $\Delta_{b}$
are defined by
\begin{equation}
	\Box_{b}
	\coloneqq \delb_{b}^{\ast} \delb_{b},
	\qquad
	\Delta_{b}
	\coloneqq \Box_{b} + \overline{\Box}_{b}.
\end{equation}
Note that
\begin{equation}
	\Box_{b}
	= \frac{1}{2} \Delta_{b} + \frac{\sqrt{- 1}}{2} n T;
\end{equation}
see~\cite{Lee1986}*{Theorem  2.3} for example.
The Gaffney extension of the Kohn Laplacian,
also denoted by $\Box_{b}$,
is a self-adjoint operator on $L^{2}(M)$.
The kernel $\Ker \Box_{b}$ is the space of $L^{2}$ CR holomorphic functions.

The \emph{critical CR GJMS operator} $\GJMS$
is a real differential operator of order $2 n + 2$ acting on $C^{\infty}(M)$.
For $\mu \in \bbR$,
set
\begin{equation}
	L_{\mu}
	\coloneqq \frac{1}{2} \Delta_{b} + \frac{\sqrt{- 1}}{2} \mu T.
\end{equation}
The critical CR GJMS operator $\GJMS$ satisfies
\begin{equation}
	\GJMS - L_{- n} L_{- n + 2} \dotsm L_{n - 2} L_{n} \in \Hpsido^{2 n + 1}(M);
\end{equation}
see \cite{Ponge2008-Book}*{Proposition 3.5.7}.
It is known to be formally self-adjoint~\cite{Gover-Graham2005}*{Proposition 5.1}.
Moreover,
it annihilates CR pluriharmonic functions~\cite{Hirachi2014}*{Section 3.2}.

A CR manifold $(M, T^{1, 0} M)$ is said to be \emph{embeddable}
if there exists a smooth embedding of $M$ into some $\bbC^{N}$
such that $T^{1, 0} M = T^{1, 0} \bbC^{N}|_{M} \cap (T M \otimes \bbC)$.
It is known that a closed strictly pseudoconvex CR manifold $(M, T^{1, 0} M)$ is embeddable
if and only if $\Box_{b}$ has closed range~\cites{Boutet_de_Monvel1975,Kohn1986}.

\section{Model operators on the Heisenberg group}
\label{section:model-operators-on-the-Heisenberg-group}

The Heisenberg group $G$ is the Lie group with the underlying manifold $\bbC^{n} \times \bbR$
and the multiplication
\begin{equation}
	(z, t) \cdot (z^{\prime}, t^{\prime})
	\coloneqq (z + z^{\prime}, t + t^{\prime} + 2 \Im (z \cdot \ovz^{\prime})).
\end{equation}
The left translation by $(z, t)$ and the inversion on $G$
are denoted by $\ltrans_{(z, t)}$ and $\inversion$ respectively.

For $\alpha = 1, \dots , n$,
we introduce a left-invariant complex vector field $Z_{\alpha}^{0}$ by
\begin{equation}
	Z_{\alpha}^{0}
	\coloneqq \pdv{}{z^{\alpha}} + \sqrt{- 1} \ovz^{\alpha} \pdv{}{t}.
\end{equation}
The canonical CR structure $T^{1, 0} G$ is spanned by $Z_{1}^{0}, \dots , Z_{n}^{0}$.
Define a left-invariant one-form $\theta^{0}$ on $G$ by
\begin{equation}
	\theta^{0}
	\coloneqq d t + \sqrt{- 1} \sum_{\alpha = 1}^{n}
		(z^{\alpha} d \ovz^{\alpha} - \ovz^{\alpha} d z^{\alpha}).
\end{equation}
Then $\theta^{0}$ annihilates $T^{1, 0} G$
and the Levi form $\calL_{\theta^{0}}$ satisfies
$\calL_{\theta^{0}}(Z_{\alpha}^{0}, Z_{\beta}^{0}) = 2 \delta_{\alpha \beta}$;
in particular,
$\theta^{0}$ is a contact form on $G$.
The Reeb vector field $T^{0}$ coincides with $\pdvf{}{t}$.

In this setting,
the sub-Laplacian $\Delta_{b}^{0}$ with respect to $\theta^{0}$ is given by
\begin{equation}
	\Delta_{b}^{0}
	= - \frac{1}{2} \sum_{\alpha = 1}^{n} (Z_{\alpha}^{0} \ovZ_{\alpha}^{0} + \ovZ_{\alpha}^{0} Z_{\alpha}^{0}).
\end{equation}
For $\mu \in \bbR$,
define a left-invariant differential operator $L_{\mu}^{0}$ by
\begin{equation}
	L_{\mu}^{0}
	\coloneqq \frac{1}{2} \Delta_{b}^{0} + \frac{\sqrt{- 1}}{2} \mu T^{0}.
\end{equation}
The critical CR GJMS operator $\GJMS^{0}$ on $G$ is written as
\begin{equation}
	\GJMS^{0}
	= L_{- n}^{0} L_{- n + 2}^{0} \dotsm L_{n - 2}^{0} L_{n}^{0};
\end{equation}
see~\cite{Graham1984}*{Section 1}.

The Lie algebra $\frakg$ of $G$ is isomorphic to $\bbC^{n} \times \bbR$ as a linear space via
\begin{equation}
	\frakg \to \bbC^{n} \times \bbR;
	\qquad
	2 \Re \sum_{\alpha = 1}^{n} z^{\alpha} Z_{\alpha}^{0} + t T^{0}
	\mapsto (z, t).
\end{equation}
Under this identification,
the Lie bracket on $\frakg$ is given by
\begin{equation}
	\comm{(z, t)}{(z^{\prime}, t^{\prime})}
	= (0, 4 \Im(z \cdot \ovz^{\prime})).
\end{equation}
Moreover,
the exponential map $\frakg \to G$ coincides with the identity map on $\bbC^{n} \times \bbR$.
Furthermore,
the dual $\frakg^{\ast}$ of $\frakg$ is also canonically isomorphic to $\bbC^{n} \times \bbR$
as a linear space.
We denote these linear coordinates by $(\zeta, \tau)$.

For $r \in \bbR_{> 0}$,
the parabolic dilation $\dilation_{r}$ on $\bbC^{n} \times \bbR$ is defined by
\begin{equation}
	\dilation_{r}(z, t)
	= (r z, r^{2} t).
\end{equation}
This dilation defines automorphisms on $G$, $\frakg$, and $\frakg^{\ast}$,
for which we will use the same letter $\dilation_{r}$ by abuse of notation.
In what follows,
the term ``homogeneous'' is defined in terms of $\dilation_{r}$.
Set $\rho(z, t) \coloneqq (\abs{z}^{4} + t^{2})^{1 / 4}$,
which is homogeneous of degree $1$.
The unitary group $U(n)$ acts on $\bbC^{n} \times \bbR$ by
\begin{equation}
	\mu(U) (z, t)
	\coloneqq (U \cdot z, t),
\end{equation}
where $U \in U(n)$.
We will sometimes write $v$ for a point of $G$.
Denote by $d v$ the Lebesgue measure on $G$,
which is a Haar measure on $G$.

Let $\Schwartz(G)$ (resp.\ $\Schwartz(\frakg^{\ast})$)
be the space of rapidly decreasing functions on $G$ (resp.\ $\frakg^{\ast}$),
and $\Schwartz^{\prime}(G)$ (resp.\ $\Schwartz^{\prime}(\frakg^{\ast})$)
be that of tempered distributions on $G$ (resp.\ $\frakg^{\ast}$).
We denote the pairing between $k \in \Schwartz^{\prime}(G)$ and $f \in \Schwartz(G)$ by $\coupling{k}{f}$.
The pull-back by $\dilation_{r}$
induces endomorphisms on $\Schwartz(G)$ and $\Schwartz(\frakg^{\ast})$,
and these extend to endomorphisms on $\Schwartz^{\prime}(G)$ and $\Schwartz^{\prime}(\frakg^{\ast})$.
The Fourier transform $\Fourier$ defines isomorphisms
\begin{equation}
	\Schwartz(G) \xrightarrow{\cong} \Schwartz(\frakg^{\ast}),
	\qquad
	\Schwartz^{\prime}(G) \xrightarrow{\cong} \Schwartz^{\prime}(\frakg^{\ast});
\end{equation}
in our convention,
the Fourier transform $\Fourier (f)$ of $f \in \Schwartz(G)$ is defined by
\begin{equation}
	\Fourier (f)(\zeta, \tau)
	\coloneqq \int_{G} e^{- \sqrt{- 1} (\Re(z \cdot \ovxz) + t \tau)} f(z, t) d v.
\end{equation}

Now we consider ``model operators'' of the Heisenberg calculus.
For $m \in \bbR$,
set
\begin{equation}
	\Hsymbol^{m}
	\coloneqq \Set{a \in C^{\infty}(\frakg^{\ast} \setminus \{0\}) | \dilation_{r}^{\ast} a = r^{m} a},
\end{equation}
which is the space of \emph{Heisenberg symbols} of order $m$.
Let $\scrG^{m}$ be the space of $g \in \Schwartz^{\prime}(\frakg^{\ast})$ such that
$g$ is smooth on $\frakg^{\ast} \setminus \{0\}$ and satisfies
\begin{equation}
	\dilation_{r}^{\ast} g
	= r^{m} g + (r^{m} \log r) h,
\end{equation}
where $h \in \Schwartz^{\prime}(\frakg^{\ast})$
with $\supp h \subset \{0\}$ and $\dilation_{r}^{\ast} h = r^{m} h$.
The restriction map $\scrG^{m} \to \Hsymbol^{m}$ is known to be surjective~\cite{Beals-Greiner1988}*{Proposition 15.8}.
Moreover,
the inverse Fourier transform gives an isomorphism
\begin{equation}
	\Fourier^{- 1} \colon \scrG^{m} \xrightarrow{\cong} \scrK_{- m - 2 n - 2},
\end{equation}
where $\scrK_{l}$ is the space of $k \in \Schwartz^{\prime}(G)$ such that
$k$ is smooth on $G \setminus \{0\}$ and satisfies
\begin{equation}
	\dilation_{r}^{\ast} k
	= r^{l} k + (r^{l} \log r) \psi
\end{equation}
for a homogeneous polynomial $\psi$ of degree $l$~\cite{Beals-Greiner1988}*{Proposition 15.24}.
We also introduce a function space on which Heisenberg symbols act.
Let $\Schwartz_{0}(G)$ be the space of $f \in \Schwartz(G)$ such that
\begin{equation}
	\int_{G} \psi(v) f(v) d v = 0
\end{equation}
for any polynomial $\psi$ on $G$.
This condition is equivalent to the condition that
$\Fourier (f) \in \Schwartz(\frakg^{\ast})$ vanishes to infinite order at the origin.

Let $\Hpsido^{m}$ denote the space of continuous linear endomorphisms $A$ of $\Schwartz_{0}(G)$
that commute with left translations
and admit a formal adjoint $A^{\ast} \colon \Schwartz_{0}(G) \to \Schwartz_{0}(G)$ satisfying
\begin{equation}
	A^{\ast} \circ \dilation_{r}^{\ast}
	= r^{m} \dilation_{r}^{\ast} \circ A^{\ast}.
\end{equation}
We now recall a canonical isomorphism between $\Hsymbol^{m}$ and $\Hpsido^{m}$.

\begin{proposition}[\cite{Takeuchi2023-GJMS}*{Proposition 3.1}]
\label{prop:well-defined-of-Hpsido}
	Let $a \in \Hsymbol^{m}$ and
	take $g \in \scrG^{m}$ with $g |_{\frakg^{\ast} \setminus \{0\}} = a$.
	Then the convolution operator
	\begin{equation}
	\label{eq:definition-of-O(p)}
		f \mapsto [\Fourier^{- 1}(g) \ast f](v)
			\coloneqq \coupling{\Fourier^{- 1}(g)}{f \circ \ltrans_{v} \circ \inversion}
	\end{equation}
	defines an endomorphism on $\Schwartz_{0}(G)$
	and is independent of the choice of $g$.
	Moreover,
	this operator commutes with left translation and is homogeneous of degree $m$.
	Furthermore,
	it is equal to zero if and only if $a = 0$.
\end{proposition}

\begin{definition}
\label{def:action-of-principal-symnol}
	For $a \in \Hsymbol^{m}$,
	an operator $O^{0}(a) \colon \Schwartz_{0}(G) \to \Schwartz_{0}(G)$
	is defined by \cref{eq:definition-of-O(p)}.
\end{definition}

The map $O^{0}$ is compatible with taking formal adjoints and with composition.

\begin{theorem}[\cite{Takeuchi2023-GJMS}*{Theorem 3.3}]
\label{thm:adjoint-and-composition-of-model-operators}
	(i) The formal adjoint of $O^{0}(a)$,
	$a \in \Hsymbol^{m}$,
	is given by $O^{0}(\ova)$.
	In particular,
	$O^{0}(a)$ is formally self-adjoint if and only if $a$ is real-valued.
	
	(ii) There exists a bilinear product
	\begin{equation}
		\ast^{0} \colon \Hsymbol^{m_{1}} \times \Hsymbol^{m_{2}} \to \Hsymbol^{m_{1} + m_{2}}
	\end{equation}
	such that $O^{0}(a_{1}) O^{0}(a_{2}) = O^{0}(a_{1} \ast^{0} a_{2})$
	for any $a_{1} \in \Hsymbol^{m_{1}}$ and $a_{2} \in \Hsymbol^{m_{2}}$.
\end{theorem}

To simplify notation,
we write
\begin{equation}
	a^{\ast^{0} k}
	\coloneqq \underbrace{a \ast^{0} \dots \ast^{0} a}_{\text{$k$ factors}}.
\end{equation}
In particular,
$O^{0}$ defines an injective map from $\Hsymbol^{m}$ to $\Hpsido^{m}$.
In fact,
this is an isomorphism.

\begin{proposition}[\cite{Takeuchi2023-GJMS}*{Proposition 3.4}]
	For any $A \in \Hpsido^{m}$,
	there exists a unique $a \in \Hsymbol^{m}$ such that $A = O^{0}(a)$.
\end{proposition}

\begin{definition}
	The \emph{Heisenberg symbol}
	\begin{equation}
		\sigma_{m}^{0} \colon \Hpsido^{m} \to \Hsymbol^{m}
	\end{equation}
	is defined as the inverse of $O^{0}$.
\end{definition}

It follows from \cref{thm:adjoint-and-composition-of-model-operators} that
\begin{equation}
	\sigma_{m}^{0}(A^{\ast})
	= \overline{\sigma_{m}^{0}(A)},
	\qquad
	\sigma_{m_{1} + m_{2}}^{0}(A_{1} A_{2})
	= \sigma_{m_{1}}^{0}(A_{1}) \ast^{0} \sigma_{m_{2}}^{0}(A_{2})
\end{equation}
for $A \in \Hpsido^{m}$, $A_{1} \in \Hpsido^{m_{1}}$, and $A_{2} \in \Hpsido^{m_{2}}$.
In particular,
$A$ is formally self-adjoint if and only if $\sigma_{m}^{0}(A)$ is real-valued.

\section{A square-root construction on the Heisenberg group}
\label{section:square-root-on-Heisenberg-group}

In this section,
we construct the square root of a modified critical CR GJMS operator on the Heisenberg group.
This square root plays a crucial role in our proofs.

To this end,
we first consider the \Szego projection $S^{0}$
and the orthogonal projection $\Pi^{0}$ onto the $L^{2}$-closure
of the complexified space of square-integrable smooth CR pluriharmonic functions on $G$.
For $\varepsilon > 0$,
set
\begin{equation}
	s_{\varepsilon}(z, t)
	\coloneqq \frac{2^{n - 1} n!}{\pi^{n + 1}} (\abs{z}^{2} + \varepsilon - \sqrt{- 1} t)^{- n - 1}
	\in C^{\infty}(G).
\end{equation}
It is known that the \Szego projection $S^{0}$ is
the convolution operator with kernel
\begin{equation}
	s \coloneqq \lim_{\varepsilon \to 0} s_{\varepsilon} \in \scrK_{- 2 n - 2},
\end{equation}
where the limit is taken in $\Schwartz^{\prime}(G)$;
see \cite{Greiner-Kohn-Stein1975}*{Equation 10}.
Moreover,
the operator $\Pi^{0}$ coincides with $S^{0} + \ovS^{0}$~\cite{Graham1984}*{Remark 2}.

\begin{theorem}
\label{thm:square-root-of-critical-CR-GJMS-operator}
	Set
	\begin{equation}
		\wtP^{0}
		\coloneqq \GJMS^{0} + \Pi^{0} (\Delta_{b}^{0})^{n + 1} \Pi^{0}.
	\end{equation}
	There exists a $U(n)$-invariant real-valued $q^{0} \in \Hsymbol^{n + 1}$
	such that $\sigma_{2 n + 2}^{0}(\wtP^{0}) = q^{0} \ast^{0} q^{0}$.
\end{theorem}

The remainder of this section is devoted to the proof of
\cref{thm:square-root-of-critical-CR-GJMS-operator}.
To simplify notation, we omit the superscript $0$ for objects on $G$.

The operators $\Delta_{b}$ and $- \sqrt{- 1} T$,
initially defined on $\Schwartz(G)$,
are essentially self-adjoint,
and their self-adjoint closures strongly commute.
We denote these closures by the same symbols.
Their joint spectrum is the subset $\Sigma = \Sigma_{1} \cup \Sigma_{2}$ of $\bbR^{2}$,
where
\begin{equation}
	\Sigma_{1}
	\coloneqq \Set{((2 k + n) \abs{\lambda}, \lambda) | k \in \bbZ_{\geq 0}, \lambda \in \bbR^{\times}}
\end{equation}
and
\begin{equation}
	\Sigma_{2}
	\coloneqq \Set{(\mu, 0) | \mu \in \clop{0}{\infty}}.
\end{equation}
Define the dilation $d_{r} \colon \bbR^{2} \to \bbR^{2}$ by
\begin{equation}
	d_{r}(\mu, \lambda) \coloneqq (r^{2} \mu, r^{2} \lambda)
\end{equation}
for any $r \in \bbR_{> 0}$.
Note that $\Sigma_{1}$ and $\Sigma_{2}$ are invariant under this dilation.
Denote by $\calE$ the joint resolution of the identity of $\Delta_{b}$ and $- \sqrt{- 1} T$,
whose support coincides with $\Sigma$.
For $f, g \in L^{2}(G)$,
let $\calE_{f, g}$ denote the complex Borel measure on $\bbR^{2}$
defined by $\calE_{f, g}(E) = (\calE(E) f, g)_{L^{2}}$.
If $\eta$ is a Borel function on $\Sigma$,
then we can define a possibly unbounded operator on $L^{2}(G)$ by
\begin{equation}
	\eta(\Delta_{b}, - \sqrt{- 1} T)
	\coloneqq \int_{\Sigma} \eta(\mu, \lambda) \, d \calE(\mu, \lambda),
\end{equation}
whose domain is given by
\begin{equation}
	\Set{f \in L^{2}(G) | \int_{\Sigma} \abs{\eta(\mu, \lambda)}^{2}
		\, d \calE_{f, f}(\mu, \lambda) < \infty}.
\end{equation}
Set $\calR \coloneqq \Delta_{b}^{2} - T^{2}$;
this satisfies
\begin{equation}
	\calR
	= \int_{\Sigma} (\mu^{2} + \lambda^{2}) \, d \calE(\mu, \lambda).
\end{equation}

\begin{lemma}
\label{lem:domain-of-operators}
	Assume that there exist $C > 0$ and $N \in \bbZ_{\geq 0}$ such that
	$\abs{\eta(\mu, \lambda)} \leq C (1 + \mu^{2} + \lambda^{2})^{N}$.
	Then the domain of $\eta(\Delta_{b}, - \sqrt{- 1} T)$ contains $\Schwartz(G)$.
\end{lemma}

\begin{proof}
	Take any $f \in \Schwartz(G)$.
	Since $(1 + \calR)^{N} f \in L^{2}(G)$,
	we have
	\begin{equation}
		\int_{\Sigma} (1 + \mu^{2} + \lambda^{2})^{2 N} \, d \calE_{f, f}(\mu, \lambda) < \infty.
	\end{equation}
	This implies
	\begin{equation}
		\int_{\Sigma} \abs{\eta(\mu, \lambda)}^{2} \, d \calE_{f, f}(\mu, \lambda)
		\leq C^{2} \int_{\Sigma} (1 + \mu^{2} + \lambda^{2})^{2 N} \, d \calE_{f, f}(\mu, \lambda) < \infty,
	\end{equation}
	which means that $f$ is an element of the domain of $\eta(\Delta_{b}, - \sqrt{- 1} T)$.
\end{proof}

In what follows,
we assume that $\eta$ satisfies the assumption of \cref{lem:domain-of-operators}.
The preceding estimate shows that $\eta(\Delta_{b}, - \sqrt{- 1} T)$ is a continuous map
from $\Schwartz(G)$ to $L^{2}(G)$.
Since it commutes with left translations,
it has a unique convolution kernel $K_{\eta} \in \Schwartz^{\prime}(G)$.
Moreover,
the operator $\eta(\Delta_{b}, - \sqrt{- 1} T)$ is $U(n)$-equivariant;
this follows from $\comm{\Delta_{b}}{\mu(U)^{\ast}} = \comm{T}{\mu(U)^{\ast}} = 0$ for any $U \in U(n)$.

\begin{lemma}
\label{lem:homogeneity-of-convolution-kernel}
	The convolution kernel $K_{\eta \circ d_{r}}$ of $(\eta \circ d_{r})(\Delta_{b}, - \sqrt{- 1} T)$
	coincides with $r^{- 2 n - 2} \dilation_{1 / r}^{\ast} K_{\eta}$.
\end{lemma}

\begin{proof}
	Let $D_{r}$ denote the unitary operator $r^{n + 1} \dilation_{r}^{\ast}$ on $L^{2}(G)$.
	Since $\Delta_{b} D_{r} = r^{2} D_{r} \Delta_{b}$ and $T D_{r} = r^{2} D_{r} T$,
	we have
	\begin{equation}
		(\eta \circ d_{r})(\Delta_{b}, - \sqrt{- 1} T)
		= D_{r}^{- 1} \eta(\Delta_{b}, - \sqrt{- 1} T) D_{r}.
	\end{equation}
	Take $f \in \Schwartz(G)$.
	Then
	\begin{align}
		(D_{r}^{- 1} \eta(\Delta_{b}, - \sqrt{- 1} T) D_{r} f)(v)
		&= r^{- n - 1} (\eta(\Delta_{b}, - \sqrt{- 1} T) D_{r} f)(\dilation_{1 / r} v) \\
		&= \coupling{K_{\eta}}{f \circ \dilation_{r} \circ l_{\dilation_{1 / r} v} \circ \inversion} \\
		&= \coupling{K_{\eta}}{f \circ l_{v} \circ \inversion \circ \dilation_{r}} \\
		&= \coupling{r^{- 2 n - 2} \dilation_{1 / r}^{\ast} K_{\eta}}{f \circ l_{v} \circ \inversion}.
	\end{align}
	The uniqueness of the convolution kernel implies
	\begin{equation}
		K_{\eta \circ d_{r}} = r^{- 2 n - 2} \dilation_{1 / r}^{\ast} K_{\eta},
	\end{equation}
	which completes the proof.
\end{proof}

Now consider the operator $\wtP$.
Note that $\wtP = \pi(\Delta_{b}, - \sqrt{- 1} T)$,
where
\begin{equation}
	\pi(\mu, \lambda)
	\coloneqq
	\begin{cases}
		2^{- n - 1} \prod_{l = 0}^{n} (\mu - (n - 2 l) \abs{\lambda}) & \mu \neq n \abs{\lambda}, \\
		\mu^{n + 1} & \mu = n \abs{\lambda}
	\end{cases}
\end{equation}
for $(\mu, \lambda) \in \Sigma$.
The projection $\Pi = S + \ovS$ selects the lowest spectral rays $\mu = n \abs{\lambda}$,
on which the multiplier of $\GJMS$ vanishes.
Thus the correction term $\Pi \Delta_{b}^{n + 1} \Pi$ replaces
the zero multiplier on these rays by $\mu^{n + 1}$.
We extend $\pi|_{\Sigma \setminus \{0\}}$ to a smooth positive function on $\bbR^{2} \setminus \{0\}$, homogeneous of degree $2 (n + 1)$ with respect to $d_{r}$,
and denote the extension by the same symbol $\pi$.
The lowest spectral rays $\mu = n \abs{\lambda}$ are separated
on the unit circle from the remaining spectral directions,
which satisfy $\mu \geq (n + 2) \abs{\lambda}$.
Since the polynomial multiplier is strictly positive near the latter set,
a smooth positive homogeneous extension can be constructed using a partition of unity on the unit circle.
Set $\kappa(\mu, \lambda) \coloneqq \sqrt{\pi(\mu, \lambda)}$,
so that $\kappa \circ d_{r} = r^{n + 1} \kappa$.
We extend $\kappa$ continuously to all of $\bbR^{2}$ by $\kappa(0) = 0$.

\begin{proposition}
\label{prop:convolution-kernel-of-square-root}
	The convolution kernel $K_{\kappa}$ of $\kappa(\Delta_{b}, - \sqrt{- 1} T)$
	is homogeneous of degree $- 3 n - 3$
	and smooth outside the origin.
\end{proposition}

\begin{proof}
	Since $\kappa \circ d_{r} = r^{n + 1} \kappa$ for any $r > 0$,
	it follows from \cref{lem:homogeneity-of-convolution-kernel} that
	\begin{equation}
		\dilation_{r}^{\ast} K_{\kappa}
		= r^{- 2 n - 2} K_{\kappa \circ d_{1 / r}}
		= r^{- 3 n - 3} K_{\kappa},
	\end{equation}
	which means that $K_{\kappa}$ is homogeneous of degree $- 3 n - 3$.
	
	It remains to show that $K_{\kappa}$ is smooth outside the origin.
	To this end,
	we construct a sequence $\gamma_{N} \in \Schwartz(G)$
	such that $(\gamma_{N})_{N = 0}^{\infty}$ converges to $K_{\kappa}$ in $\Schwartz^{\prime}(G)$
	and in $C^{\infty}_{\mathrm{loc}}(G \setminus \{0\})$.
	Take a non-negative $\chi \in C^{\infty}_{c}(\bbR^{2} \setminus \{0\})$ so that
	\begin{equation}
		\sum_{j \in \bbZ} \chi(2^{- 2 j} \xi) = 1
	\end{equation}
	for any $\xi \in \bbR^{2} \setminus \{0\}$.
	Set $\eta(\xi) \coloneqq \kappa(\xi) \chi(\xi) \in C^{\infty}_{c}(\bbR^{2} \setminus \{0\})$ and
	$k \coloneqq K_{\eta}$.
	Since $\eta \in C^{\infty}_{c}(\bbR^{2})$,
	the Schwartz multiplier theorem for the joint functional calculus of $\Delta_{b}$ and $- \sqrt{- 1} T$
	implies that $k \in \Schwartz(G)$;
	see \cite{Astengo-Di_Blasio-Ricci2009}*{Corollary 5.4}.
	Consider
	\begin{equation}
		\eta_{j}(\xi)
		\coloneqq \kappa(\xi) \chi(2^{- 2 j} \xi)
		= 2^{(n + 1) j} \kappa(2^{- 2 j} \xi) \chi(2^{- 2 j} \xi)
		= 2^{(n + 1) j} \eta(2^{- 2 j} \xi).
	\end{equation}
	This formula and \cref{lem:homogeneity-of-convolution-kernel} imply that
	the convolution kernel $k_{j}$ of $\eta_{j}(\Delta_{b}, - \sqrt{- 1} T)$
	coincides with $2^{(3 n + 3) j} k \circ \dilation_{2^{j}}$.
	Define
	\begin{equation}
		\gamma_{N}
		\coloneqq \sum_{j = - N}^{N} k_{j}
		= \sum_{j = - N}^{N} 2^{(3 n + 3) j} k \circ \dilation_{2^{j}}
		\in \Schwartz(G).
	\end{equation}
	
	We first show that $(\gamma_{N})_{N = 0}^{\infty}$ converges in $\Schwartz^{\prime}(G)$.
	Take $f \in \Schwartz(G)$.
	If $j < 0$,
	then
	\begin{equation}
		\abs{\coupling{k_{j}}{f}}
		= 2^{(3 n + 3) j} \abs*{\int_{G} k \circ \dilation_{2^{j}}(v) f(v) \, d v}
		\leq 2^{(3 n + 3) j} \norm{k}_{L^{\infty}} \norm{f}_{L^{1}}.
	\end{equation}
	This implies that $\sum_{l = 1}^{\infty} k_{- l}$ converges in $\Schwartz^{\prime}(G)$.
	If $j \geq 0$,
	consider
	\begin{equation}
		\nu(\mu, \lambda)
		\coloneqq \frac{\eta(\mu, \lambda)}{(\mu^{2} + \lambda^{2})^{n + 1}}
		\in C^{\infty}_{c}(\bbR^{2} \setminus \{0\}).
	\end{equation}
	The convolution kernel $k^{\prime} \coloneqq K_{\nu} \in \Schwartz(G)$ satisfies
	$k = \calR^{n + 1} k^{\prime}$.
	Then
	\begin{align}
		k_{j}
		&= 2^{(3 n + 3) j} k \circ \dilation_{2^{j}} \\
		&= 2^{(3 n + 3) j} (\calR^{n + 1} k^{\prime}) \circ \dilation_{2^{j}} \\
		&= 2^{- (n + 1) j} \calR^{n + 1} (k^{\prime} \circ \dilation_{2^{j}}).
	\end{align}
	Integration by parts yields
	\begin{align}
		\abs{\coupling{k_{j}}{f}}
		&= 2^{- (n + 1) j} \abs*{\int_{G} k^{\prime} \circ \dilation_{2^{j}}(v) (\calR^{n + 1} f)(v) \, d v} \\
		&\leq 2^{- (n + 1) j} \norm{k^{\prime}}_{L^{\infty}} \norm{\calR^{n + 1} f}_{L^{1}},
	\end{align}
	which implies that $\sum_{j = 0}^{\infty} k_{j}$ converges in $\Schwartz^{\prime}(G)$.
	Therefore $\gamma_{N} = \sum_{j = - N}^{N} k_{j}$ also converges
	to some $K$ in $\Schwartz^{\prime}(G)$.
	
	We next show that the limit $K$ of $(\gamma_{N})_{N = 0}^{\infty}$ coincides with $K_{\kappa}$.
	To this end,
	consider
	\begin{equation}
		\kappa_{N}(\xi)
		= \sum_{j = - N}^{N} \eta_{j}(\xi)
		= \kappa(\xi) \sum_{j = - N}^{N} \chi(2^{- 2 j} \xi).
	\end{equation}
	Note that $\gamma_{N} = K_{\kappa_{N}}$.
	Since $0 \leq \kappa_{N} \leq \kappa$
	and $(\kappa_{N})_{N = 0}^{\infty}$ converges to $\kappa$ pointwise,
	the dominated convergence theorem implies
	\begin{equation}
		\norm{(\kappa - \kappa_{N})(\Delta_{b}, - \sqrt{- 1} T) f}_{L^{2}}^{2}
		= \int_{\Sigma} \abs{\kappa(\xi) - \kappa_{N}(\xi)}^{2} \, d \calE_{f, f}(\xi)
		\to 0
	\end{equation}
	as $N \to \infty$ for any $f \in \Schwartz(G)$.
	Thus we have
	\begin{equation}
		\gamma_{N} \ast f = \kappa_{N}(\Delta_{b}, - \sqrt{- 1} T) f
		\to \kappa(\Delta_{b}, - \sqrt{- 1} T) f = K_{\kappa} \ast f
		\quad \text{ in } L^{2}(G).
	\end{equation}
	On the other hand,
	$\gamma_{N} \to K$ in $\Schwartz^{\prime}(G)$ gives
	\begin{equation}
		\gamma_{N} \ast f \to K \ast f
		\quad \text{ in } \Schwartz^{\prime}(G).
	\end{equation}
	By uniqueness of the convolution kernel,
	$K = K_{\kappa}$.
	
	It remains to prove that $\gamma_{N}$ converges to $K_{\kappa}$
	in $C^{\infty}_{\mathrm{loc}}(G \setminus \{0\})$.
	Fix a left-invariant differential operator $D$ on $G$ homogeneous of degree $d$.
	Since $D k \in \Schwartz(G)$,
	there exists a positive constant $C$ such that
	\begin{equation}
		\abs{(D k)(v)} \leq C (1 + \rho(v))^{- 4 n - 4 - d}.
	\end{equation}
	It follows from the definition of $k_{j}$ that
	\begin{equation}
		D k_{j}
		= 2^{(3 n + 3 + d) j} (D k) \circ \dilation_{2^{j}},
	\end{equation}
	which implies
	\begin{equation}
		\abs{(D k_{j})(v)}
		\leq C 2^{(3 n + 3 + d) j} (1 + 2^{j} \rho(v))^{- 4 n - 4 - d}.
	\end{equation}
	Now fix a compact set $E$ in $G \setminus \{0\}$
	and put $\varepsilon \coloneqq \inf_{v \in E} \rho(v) > 0$.
	If $j < 0$,
	then
	\begin{equation}
		\abs{(D k_{j})(v)} \leq C 2^{(3 n + 3 + d) j}.
	\end{equation}
	It follows that $\sum_{l = 1}^{\infty} D k_{- l}$ converges uniformly on $E$.
	On the other hand,
	if $j \geq 0$,
	then
	\begin{equation}
		\abs{(D k_{j})(v)}
		\leq C 2^{(3 n + 3 + d) j} (2^{j} \rho(v))^{- 4 n - 4 - d}
		\leq C \varepsilon^{- 4 n - 4 - d} 2^{- (n + 1) j}.
	\end{equation}
	Hence $\sum_{j = 0}^{\infty} D k_{j}$ also converges uniformly on $E$.
	Therefore $\gamma_{N} = \sum_{j = - N}^{N} k_{j}$
	converges to $K_{\kappa}$ in $C^{\infty}_{\mathrm{loc}}(G \setminus \{0\})$.
\end{proof}

\begin{proof}[Proof of \cref{thm:square-root-of-critical-CR-GJMS-operator}]
	It follows from \cref{prop:convolution-kernel-of-square-root} that
	$K_{\kappa} \in \scrK_{- 3 n - 3}$.
	Set $q = \Fourier(K_{\kappa})|_{\frakg^{\ast} \setminus \{0\}} \in \Hsymbol^{n + 1}$.
	By \cref{prop:well-defined-of-Hpsido} and \cref{def:action-of-principal-symnol},
	$O(q)$ coincides with $\kappa(\Delta_{b}, - \sqrt{- 1} T)$ on $\Schwartz_{0}(G)$.
	Since $\kappa(\Delta_{b}, - \sqrt{- 1} T)$ is formally self-adjoint and $U(n)$-equivariant,
	$q$ is real-valued and $U(n)$-invariant.
	Moreover,
	$\wtP = \kappa(\Delta_{b}, - \sqrt{- 1} T)^{2}$ implies $\sigma_{2 n + 2}(\wtP) = q \ast q$.
\end{proof}

\section{Heisenberg calculus and a lower-bound criterion}
\label{section:Heisenberg-calculus-and-lower-bound-criterion}

In this section,
we recall basic properties of Heisenberg pseudodifferential operators;
see~\cites{Beals-Greiner1988,Ponge2008-Book}
for a comprehensive introduction to the Heisenberg calculus.

Throughout this section,
we fix a closed pseudo-Hermitian manifold $(M, T^{1, 0} M, \theta)$ of dimension $2 n + 1$.
Let
\begin{equation}
	\frakg M
	\coloneqq H M \oplus (T M / H M).
\end{equation}
The Reeb vector field $T$ defines a nowhere-vanishing section $[T]$ of $T M / H M$.
For sections $X^{\prime}$ and $Y^{\prime}$ of $H M$
and $X_{0}$ and $Y_{0}$ of $T M / H M$,
the Lie bracket $\comm{X^{\prime} + X_{0}}{Y^{\prime} + Y_{0}}$ is defined by
\begin{equation}
	\comm{X^{\prime} + X_{0}}{Y^{\prime} + Y_{0}}
	\coloneqq - d \theta (X^{\prime}, Y^{\prime}) [T].
\end{equation}
This bracket makes $\frakg M$ a bundle of two-step nilpotent Lie algebras.
The dilation $\dilation_{r}$ on $\frakg M$ is defined by
\begin{equation}
	\dilation_{r}(X^{\prime} + X_{0})
	= r X^{\prime} + r^{2} X_{0}.
\end{equation}
It follows from the definition of the Lie bracket that
$\dilation_{r}$ is a fiberwise Lie algebra isomorphism.
Set $G M \coloneqq \frakg M$ as a smooth fiber bundle
with the fiberwise group structure defined via the Baker-Campbell-Hausdorff formula.
The dilation $\dilation_{r}$ on $\frakg M$ induces that on $G M$,
which we also denote by $\dilation_{r}$.

Take a local frame $(Z_{\alpha})$ of $T^{1, 0} M$ on an open set $U \subset M$
such that
\begin{equation}
	\calL_{\theta}(Z_{\alpha}, Z_{\beta}) = 2 \tensor{\dilation}{_{\alpha}_{\beta}}.
\end{equation}
Then the map
\begin{equation}
\label{eq:identification-with-trivial-Lie-alg-bundle}
	\frakg M |_{U} \to U \times \frakg;
	\qquad \rbra*{ p, 2 \Re \sum_{\alpha = 1}^{n} z^{\alpha} Z_{\alpha} + t [T]}
		\mapsto (p, z, t)
\end{equation}
defines an isomorphism of Lie algebra bundles.
This isomorphism is compatible with the dilation.
The identification \cref{eq:identification-with-trivial-Lie-alg-bundle} induces those on $G M$
and the dual bundle $\frakg^{\ast} M \coloneqq (\frakg M)^{\ast}$ of $\frakg M$:
\begin{equation}
\label{eq:identification-with-trivial-Lie-grp-bundle}
	G M |_{U} \to U \times G,
	\qquad
	\frakg^{\ast} M |_{U} \to U \times \frakg^{\ast}.
\end{equation}
These are also compatible with the dilation.
Let $(Z_{\alpha}^{\prime})$ be another local frame of $T^{1, 0} M$ on $U$ satisfying
$\calL_{\theta}(Z_{\alpha}^{\prime}, Z_{\beta}^{\prime}) = 2 \tensor{\delta}{_{\alpha}_{\beta}}$.
This gives another identification $\frakg M |_{U} \to U \times \frakg$.
These two identifications are related by a smooth family $(U(p))_{p \in U}$ of unitary matrices;
that is,
\begin{equation}
	U \times \frakg \to U \times \frakg;
	\qquad
	(p, z, t)
	\mapsto (p, U(p) \cdot z, t).
\end{equation}
The same is true for $G M$ and $\frakg^{\ast} M$.

For $m \in \bbR$,
the space $\Hsymbol^{m}(M)$
consists of functions in $C^{\infty}(\frakg^{*} M \setminus \{0\})$
that are homogeneous of degree $m$ on each fiber.
Under the identification \cref{eq:identification-with-trivial-Lie-grp-bundle},
the fiberwise product $\ast^{0}$ induces a well-defined bilinear product
\begin{equation}
	\ast \colon \Hsymbol^{m_{1}}(M) \times \Hsymbol^{m_{2}}(M)
	\to \Hsymbol^{m_{1} + m_{2}}(M).
\end{equation}
To simplify notation,
we write
\begin{equation}
	a^{\ast k}
	\coloneqq \underbrace{a \ast \dots \ast a}_{\text{$k$ factors}}.
\end{equation}

Now we consider Heisenberg pseudodifferential operators.
For $m \in \bbR$,
denote by $\Hpsido^{m}(M)$
the space of \emph{Heisenberg pseudodifferential operators}
\begin{equation}
	A \colon C^{\infty}(M) \to C^{\infty}(M)
\end{equation}
of order $m$.
For example,
$V \in \Gamma(H M)$ is an element of $\Hpsido^{1}(M)$
and $T \in \Hpsido^{2}(M)$.
This space is closed under complex conjugation,
transposition, and taking formal adjoints~\cite{Ponge2008-Book}*{Proposition 3.1.23}.
In particular,
any $A \in \Hpsido^{m}(M)$ extends to a linear operator
\begin{equation}
	A \colon \scrD^{\prime}(M) \to \scrD^{\prime}(M),
\end{equation}
where $\scrD^{\prime}(M)$ is the space of distributions on $M$.
Note that
\begin{equation}
	\Hpsido^{- \infty}(M) \coloneqq \bigcap_{m \in \bbR} \Hpsido^{m}(M)
\end{equation}
coincides with the space of smoothing operators on $M$.
As in the usual pseudodifferential calculus,
there exists the \emph{Heisenberg principal symbol}
\begin{equation}
	\sigma_{m} \colon \Hpsido^{m}(M) \to \Hsymbol^{m}(M),
\end{equation}
which has the following properties:

\begin{proposition}[\cite{Ponge2008-Book}*{Propositions 3.2.6, 3.2.9, and 3.2.12}]
\label{prop:Heisenberg-principal-symbol}
	(i) The Heisenberg principal symbol $\sigma_{m}$ gives the following exact sequence:
	\begin{equation}
		0 \to \Hpsido^{m - 1}(M) \to \Hpsido^{m}(M) \xrightarrow{\sigma_{m}} \Hsymbol^{m}(M) \to 0.
	\end{equation}
	
	(ii) The formal adjoint $A^{\ast}$ of $A \in \Hpsido^{m}(M)$ satisfies
	\begin{equation}
		\sigma_{m}(A^{\ast}) = \overline{\sigma_{m}(A)}.
	\end{equation}
	In particular,
	if $A$ is formally self-adjoint,
	then $\sigma_{m}(A)$ is real-valued.
	
	(iii) For $A_{1} \in \Hpsido^{m_{1}}(M)$ and $A_{2} \in \Hpsido^{m_{2}}(M)$,
	the operator $A_{1} A_{2}$ is a Heisenberg pseudodifferential operator of order $m_{1} + m_{2}$,
	and
	\begin{equation}
		\sigma_{m_{1} + m_{2}}(A_{1} A_{2}) = \sigma_{m_{1}}(A_{1}) \ast \sigma_{m_{2}}(A_{2}).
	\end{equation}
\end{proposition}

Next,
consider approximate inverses of Heisenberg pseudodifferential operators.
We write $A \sim B$ if $A - B$ is a smoothing operator.

\begin{definition}
	Let $A \in \Hpsido^{m}(M)$.
	An operator $B \in \Hpsido^{- m}(M)$ is called a \emph{parametrix} of $A$
	if $A B \sim I$ and $B A \sim I$.
\end{definition}

The existence of a parametrix of a Heisenberg pseudodifferential operator
is determined only by its Heisenberg principal symbol.

\begin{proposition}[\cite{Ponge2008-Book}*{Proposition 3.3.1}]
\label{prop:equivalent-conditions-for-existence-of-parametrix}
	Let $A \in \Hpsido^{m}(M)$
	with Heisenberg principal symbol $a \in \Hsymbol^{m}(M)$.
	Then the following are equivalent:
	\begin{enumerate}
		\item $A$ has a parametrix;
		\item there exists $B \in \Hpsido^{- m}(M)$ such that
			$A B - I, B A - I \in \Hpsido^{- 1}(M)$;
		\item there exists $b \in \Hsymbol^{- m}(M)$ such that
			$a \ast b = b \ast a = 1$.
	\end{enumerate}
\end{proposition}

Now consider the Heisenberg differential operator $\Delta_{b} + 1$ of order $2$.
It is known that this operator has a parametrix;
see the proof of~\cite{Ponge2008-Book}*{Proposition 3.5.7} for example.
Since $\Delta_{b} + 1$ is positive and self-adjoint,
the $s$-th power $(\Delta_{b} + 1)^{s}$ of $\Delta_{b} + 1$,
$s \in \mathbb{R}$,
is a Heisenberg pseudodifferential operator of order $2 s$~\cite{Ponge2008-Book}*{Theorems 5.3.1 and 5.4.10}.
Using this operator,
we define
\begin{equation}
	\HSobolev^{s}(M)
	:= \Set{ u \in \scrD^{\prime}(M) \mid (\Delta_{b} + 1)^{s / 2} u \in L^{2}(M) }.
\end{equation}
This space is a Hilbert space with the inner product
\begin{equation}
	\iproduct{u}{v}_{s}
	= \iproduct{(\Delta_{b} + 1)^{s / 2} u}{(\Delta_{b} + 1)^{s / 2} v}_{L^{2}};
\end{equation}
write $\norm{\cdot}_{s}$ for the norm determined by $\iproduct{\cdot}{\cdot}_{s}$.
The space $C^{\infty}(M)$ is dense in $\HSobolev^{s}(M)$,
and $C^{\infty}(M) = \bigcap_{s \in \bbR} \HSobolev^{s}(M)$~\cite{Ponge2008-Book}*{Proposition 5.5.3}.
Moreover,
we obtain from the definition of $\norm{\cdot}_{s}$ that
\begin{equation}
	\abs{\iproduct{u}{v}_{0}}
	\leq \norm{u}_{s} \norm{v}_{- s}
\end{equation}
for $u, v \in C^{\infty}(M)$.
Note that,
for $k \in \bbZ_{\geq 0}$,
the Hilbert space $\HSobolev^{k}(M)$ coincides with the Folland-Stein space $S^{k, 2}(M)$
as a topological vector space~\cite{Ponge2008-Book}*{Proposition 5.5.5}.
As in the usual $L^{2}$-Sobolev space theory,
we obtain the following:

\begin{lemma}[\cite{Takeuchi2023-GJMS}*{Lemma 4.5}]
\label{lem:Rellich's-lemma}
	For $s_{1} < s_{2}$,
	the embedding $\HSobolev^{s_{2}}(M) \hookrightarrow \HSobolev^{s_{1}}(M)$ is compact.
\end{lemma}

Moreover,
the following interpolation estimate allows us
to control an intermediate-order Heisenberg Sobolev norm
by higher- and lower-order norms,
with an arbitrarily small coefficient on the higher-order term.

\begin{lemma}
\label{lem:interpolation-theorem}
	For $s_{1} < s_{2} < s_{3}$ and $\varepsilon > 0$,
	there exists a constant $C > 0$ such that
	\begin{equation}
	\label{eq:interpolation-for-Sobolev-spaces}
		\norm{u}_{s_{2}}^{2}
		\leq \varepsilon \norm{u}_{s_{3}}^{2} + C \norm{u}_{s_{1}}^{2}
	\end{equation}
	for any $u \in \HSobolev^{s_{3}}(M)$.
\end{lemma}

\begin{proof}
	The operator $(\Delta_{b} + 1)^{s^{\prime}/2}$, $s^{\prime} \in \mathbb{R}$,
	gives an isometry $\HSobolev^{s + s^{\prime}}(M) \to \HSobolev^{s}(M)$,
	and so we may assume that $s_{1} = 0$.
	Let $E$ be the resolution of the identity corresponding to
	the self-adjoint operator $\Delta_{b} + 1$.
	For $s \geq 0$ and $u \in \HSobolev^{s}(M)$,
	\begin{equation}
		\norm{u}_{s}^{2}
		= \int_{\clop{1}{\infty}} \lambda^{s} \, d E_{u, u}(\lambda),
	\end{equation}
	where $E_{u, u}$ is the Borel measure on $\bbR$ defined by $E_{u, u}(A) = (E(A) u, u)_{0}$.
	On the other hand,
	for any $\varepsilon > 0$,
	we can take a constant $C > 0$ such that
	\begin{equation}
		\lambda^{s_{2}}
		\leq \varepsilon \lambda^{s_{3}} + C
	\end{equation}
	for all $\lambda \geq 1$.
	Integrating this inequality against $E_{u, u}$ proves \eqref{eq:interpolation-for-Sobolev-spaces}.
\end{proof}

Heisenberg pseudodifferential operators act on these Hilbert spaces as follows:

\begin{proposition}[\cite{Ponge2008-Book}*{Propositions 5.5.8 and 5.5.9}]
\label{prop:mapping-properties-of-Hpsido}
	Any $A \in \Hpsido^{m}(M)$
	extends to a continuous linear operator
	\begin{equation}
		A \colon \HSobolev^{s + m}(M) \to \HSobolev^{s}(M)
	\end{equation}
	for every $s \in \bbR$.
	Moreover,
	if $m \geq 0$ and $A$ has a parametrix,
	then there exists a positive constant $C$ for each $s$ such that
	\begin{equation}
		\norm{u}_{s + m}^{2}
		\leq C (\norm{A u}_{s}^{2} + \norm{u}_{s}^{2})
	\end{equation}
	for $u \in \HSobolev^{s + m}(M)$.
\end{proposition}

We conclude this section by giving a sufficient condition
for boundedness from below of a Heisenberg pseudodifferential operator.

\begin{lemma}
\label{lem:equivalent-conditions-for-boundedness}
	Let $A$ be a formally self-adjoint Heisenberg pseudodifferential operator of order $m > 0$
	with Heisenberg principal symbol $a \in \Hsymbol^{m}(M)$.
	Then the following conditions are equivalent:
	\begin{enumerate}
		\item there exists a formally self-adjoint $B \in \Hpsido^{m / 2}(M)$ such that
			$A - B^{2} \in \Hpsido^{m - 1}(M)$;
		\item there exists a real-valued $b \in \Hsymbol^{m / 2}(M)$ such that
			$a = b \ast b$.
	\end{enumerate}
	If one of these conditions is satisfied and $A$ has a parametrix,
	then there exist $\varepsilon > 0$ and $C > 0$ such that
	\begin{equation}
	\label{eq:Garding-type-inequality}
		\iproduct{A u}{u}_{0}
		\geq \varepsilon \norm{u}_{m / 2}^{2} - C \norm{u}_{0}^{2}
	\end{equation}
	for $u \in \HSobolev^{m}(M)$.
\end{lemma}

\begin{proof}
	If (1) holds,
	then $\sigma_{m / 2}(B) \in \Hsymbol^{m / 2}(M)$ is real-valued
	and $\sigma_{m}(A) = \sigma_{m / 2}(B) \ast \sigma_{m / 2}(B)$.
	Conversely,
	assume (2).
	Then there exists $B_{0} \in \Hpsido^{m / 2}(M)$ such that
	$\sigma_{m / 2}(B_{0}) = b$.
	Consider $B = (B_{0} + B_{0}^{*}) / 2$.
	This is formally self-adjoint and $\sigma_{m / 2}(B) = b$.
	Since $\sigma_{m}(A - B^{2}) = a - b \ast b = 0$,
	we have $A - B^{2} \in \Hpsido^{m - 1}(M)$.
	This proves the equivalence between (1) and (2).
	Suppose that there exist a formally self-adjoint $B \in \Hpsido^{m / 2}(M)$
	and $R \in \Hpsido^{m - 1}(M)$
	such that $A = B^{2} + R$.
	It suffices to establish the estimate for $u \in C^{\infty}(M)$,
	since $C^{\infty}(M)$ is dense in $\HSobolev^{m}(M)$
	and both sides of the desired inequality are continuous with respect to the $\HSobolev^{m}$-topology.
	Then
	\begin{align}
		\iproduct{A u}{u}_{0}
		&= \norm{B u}_{0}^{2} + \iproduct{R u}{u}_{0} \\
		&\geq \norm{B u}_{0}^{2} - \norm{R u}_{- (m - 1) / 2} \norm{u}_{(m - 1) / 2}.
	\end{align}
	Since $A$ has a parametrix,
	its principal symbol $a = b \ast b$ is invertible with respect to the symbol product $\ast$.
	Hence $b$ is also invertible,
	and $B$ has a parametrix by \cref{prop:equivalent-conditions-for-existence-of-parametrix}.
	It follows from \cref{prop:mapping-properties-of-Hpsido} that
	there exists $\varepsilon_{1} > 0$ such that
	\begin{equation}
		\norm{B u}_{0}^{2}
		\geq \varepsilon_{1} \norm{u}_{m / 2}^{2} - \norm{u}_{0}^{2}.
	\end{equation}
	On the other hand,
	we can take $C_{1} > 0$ such that
	$\norm{R u}_{- (m - 1) / 2} \leq C_{1} \norm{u}_{(m - 1) / 2}$
	since $R \colon \HSobolev^{(m - 1) / 2}(M) \to \HSobolev^{- (m - 1) / 2}(M)$ is continuous.
	Thus we have
	\begin{equation}
		\norm{R u}_{- (m - 1) / 2} \norm{u}_{(m - 1) / 2}
		\leq C_{1} \norm{u}_{(m - 1) / 2}^{2}.
	\end{equation}
	If $0 < m \leq 1$,
	we have $\norm{u}_{(m - 1) / 2} \leq \norm{u}_{0}$.
	If $m > 1$,
	we derive from \cref{eq:interpolation-for-Sobolev-spaces} that,
	for any small $\varepsilon_{2} > 0$,
	there exists a positive constant $C_{2}$ such that
	\begin{equation}
		C_{1} \norm{u}_{(m - 1) / 2}^{2}
		\leq \varepsilon_{2} \norm{u}_{m / 2}^{2} + C_{2} \norm{u}_{0}^{2}.
	\end{equation}
	Choosing $0 < \varepsilon_{2} < \varepsilon_{1}$ completes the proof.
\end{proof}

\section{Lower bounds in the embeddable case}
\label{section:lower-bounds-in-the-embeddable}

In this section,
we prove \cref{thm:Garding-type-inequality-for-critical-CR-GJMS-operator}.
Let $(M, T^{1, 0} M, \theta)$ be a closed pseudo-Hermitian manifold of dimension $2 n + 1$.
Assume that $(M, T^{1, 0} M)$ is embeddable;
note that this assumption automatically holds if $n \geq 2$.
The author proved that
$\GJMS$ is self-adjoint and has closed range~\cite{Takeuchi2023-GJMS}*{Theorem 1.1}.
Let $G$ be the partial inverse of $\GJMS$
and let $\Pi$ be the orthogonal projection onto $\Ker \GJMS$.
The author also showed that
$G \in \Hpsido^{- 2 n - 2}(M)$ and $\Pi \in \Hpsido^{0}(M)$.
Moreover,
$\Pi$ coincides with $S + \ovS$ modulo a smoothing operator,
where $S$ is the \Szego projection~\cite{Takeuchi2023-GJMS}*{Theorem 5.6}.

In order to prove \cref{thm:Garding-type-inequality-for-critical-CR-GJMS-operator},
we introduce a formally self-adjoint Heisenberg pseudodifferential operator $\wtP$ of order $2 n + 2$ given by
\begin{equation}
	\wtP
	= \GJMS + \Pi (\Delta_{b} + 1)^{n + 1} \Pi.
\end{equation}

\begin{lemma}
\label{lem:parametrix-of-modified-GJMS-op}
	The operator $\wtP$ has a parametrix.
\end{lemma}

\begin{proof}
	From \cref{prop:equivalent-conditions-for-existence-of-parametrix},
	it is sufficient to find $\wtG_{1} \in \Hpsido^{- 2 n - 2}(M)$
	such that
	\begin{equation}
		\wtG_{1} \wtP - I, \wtP \wtG_{1} - I \in \Hpsido^{- 1}(M).
	\end{equation}
	Define $\wtG_{1}$ by
	\begin{equation}
		\wtG_{1}
		\coloneqq G + \Pi (\Delta_{b} + 1)^{- n - 1} \Pi.
	\end{equation}
	Then
	\begin{equation}
		\wtG_{1} \wtP
		= I - \Pi + \Pi (\Delta_{b} + 1)^{- n - 1} \Pi (\Delta_{b} + 1)^{n + 1} \Pi.
	\end{equation}
	We derive from $\Pi \sim S + \ovS$ and \cite{Takeuchi2023-GJMS}*{Lemma 5.1} that
	\begin{equation}
		\comm{\Pi}{(\Delta_{b} + 1)^{n + 1}}
		\in \Hpsido^{2 n + 1}(M).
	\end{equation}
	Hence
	\begin{equation}
		\wtG_{1} \wtP - I
		= \Pi (\Delta_{b} + 1)^{- n - 1} \comm{\Pi}{(\Delta_{b} + 1)^{n + 1}} \Pi
		\in \Hpsido^{- 1}(M).
	\end{equation}
	Taking the adjoint,
	we also have $\wtP \wtG_{1} - I \in \Hpsido^{- 1}(M)$.
\end{proof}

\begin{proof}[Proof of \cref{thm:Garding-type-inequality-for-critical-CR-GJMS-operator}]
	Take a local frame $(Z_{\alpha})$ of $T^{1, 0} M$ on an open set $U \subset M$
	such that $\calL_{\theta}(Z_{\alpha}, Z_{\beta}) = 2 \tensor{\delta}{_{\alpha}_{\beta}}$.
	This gives an isomorphism
	\begin{equation}
		\frakg^{\ast} M |_{U} \to U \times \frakg^{\ast}.
	\end{equation}
	Under this identification,
	it is known that
	\begin{equation}
		\sigma_{2}(L_{\mu})
		= \sigma_{2}^{0}(L_{\mu}^{0}),
		\qquad
		\sigma_{0}(S)
		= \sigma_{0}^{0}(S^{0});
	\end{equation}
	see~\cite{Ponge2008}*{Section 5}.
	Thus
	\begin{align}
		\sigma_{2 n + 2}(\wtP)
		&= \sigma_{2}(L_{- n}) \ast \dots \ast \sigma_{2}(L_{n})
			+ \sigma_{0}(\Pi) \ast \sigma_{2}(\Delta_{b})^{\ast (n + 1)} \ast \sigma_{0}(\Pi) \\
		&= \sigma_{2}^{0}(L_{- n}^{0}) \ast^{0} \dots \ast^{0} \sigma_{2}^{0}(L_{n}^{0})
			+ \sigma_{0}^{0}(\Pi^{0}) \ast^{0} \sigma_{2}^{0}(\Delta_{b}^{0})^{\ast^{0} (n + 1)}
			\ast^{0} \sigma_{0}^{0}(\Pi^{0}) \\
		&= \sigma_{2 n + 2}^{0}(\wtP^{0}).
	\end{align}
	Let $q^{0}$ be the $U(n)$-invariant real-valued symbol
	furnished by \cref{thm:square-root-of-critical-CR-GJMS-operator}.
	Under the above local identifications,
	$q^{0}$ defines local symbols that agree on overlaps,
	since the transition maps are $U(n)$-valued.
	They therefore determine a global real-valued symbol $q \in \Hsymbol^{n + 1}(M)$ satisfying
	\begin{equation}
		\sigma_{2 n + 2}(\wtP) = q * q.
	\end{equation}
	
	Let $u \in \Dom \GJMS \cap (\Ker \GJMS)^{\perp}$.
	It follows from $G \in \Hpsido^{- 2 n - 2}(M)$ that
	$u = G \GJMS u \in \HSobolev^{2 n + 2}(M)$.
	Moreover,
	$\Pi u = 0$,
	and hence $\wtP u = \GJMS u$.
	By \cref{lem:equivalent-conditions-for-boundedness,lem:parametrix-of-modified-GJMS-op},
	we have
	\begin{equation}
		(\GJMS  u, u)_{0}
		= (\wtP u, u)_{0}
		\geq \varepsilon \norm{u}_{n + 1}^{2} - C \norm{u}_{0}^{2}.
	\end{equation}
	Take a general $v \in \Dom \GJMS$.
	Then
	\begin{equation}
		(\GJMS v, v)_{0}
		= (\GJMS (I - \Pi) v, (I - \Pi) v)_{0}
		\geq - C \norm{(I - \Pi) v}_{0}^{2}
		\geq - C \norm{v}_{0}^{2},
	\end{equation}
	which means that $\GJMS$ is bounded below.
\end{proof}

\begin{corollary}
\label{cor:finite-negative-eigenvalues}
	Under the hypotheses of \cref{thm:Garding-type-inequality-for-critical-CR-GJMS-operator},
	the negative spectral subspace of $\GJMS$ is finite-dimensional.
	In particular,
	$\GJMS$ has only finitely many negative eigenvalues,
	counted with multiplicity.
\end{corollary}

\begin{proof}
	Let $W$ denote the negative spectral subspace of $\GJMS$,
	equipped with the norm inherited from $L^{2}(M)$.
	Since $\GJMS$ is bounded below,
	the spectral theorem implies that $W \subset \Dom \GJMS \cap (\Ker \GJMS)^{\perp}$.
	\cref{thm:Garding-type-inequality-for-critical-CR-GJMS-operator} implies that
	\begin{equation}
		0
		\geq (\GJMS u, u)_{0}
		\geq \varepsilon \norm{u}_{n + 1}^{2} - C \norm{u}_{0}^{2}
	\end{equation}
	for any $u \in W$.
	The inclusion $W \hookrightarrow L^{2}(M)$ therefore factors through $\HSobolev^{n + 1}(M)$,
	and is compact by \cref{lem:Rellich's-lemma}.
	Since $W$ is a closed subspace of $L^{2}(M)$,
	it must be finite-dimensional.
\end{proof}

\begin{remark}
\label{rem:negative-eigenvalue-in-higher-dimension}
	As we noted in the introduction,
	the critical CR GJMS operator $\GJMS$ is non-negative
	if $(M, T^{1, 0} M)$ is of dimension three and embeddable~\cite{Takeuchi2020-Paneitz}*{Theorem 1.1}.
	On the other hand,
	$\GJMS$ may have negative eigenvalues in higher dimensions;
	see \cite{Takeuchi2018}*{Theorem 1.1 and the proof of Theorem 1.6}.
\end{remark}

\section{The CR Paneitz operator in dimension three}
\label{section:CR-Paneitz-operator}

In this section,
we prove \cref{thm:bounded-below-in-dimension-three}.
Let $(M, T^{1, 0} M, \theta)$ be a closed pseudo-Hermitian manifold of dimension three.
We do not assume that $(M, T^{1, 0} M)$ is embeddable.
The author proved that $\GJMS$ is self-adjoint~\cite{Takeuchi2024-preprint}*{Theorem 1.1}.
Let $E$ be the resolution of the identity of $\GJMS$ and fix $\lambda > 0$.
Set
\begin{equation}
	\pi_{\lambda} \coloneqq E(\clcl{- \lambda}{\lambda})
	\colon L^{2}(M) \to \Dom \GJMS.
\end{equation}
This is an orthogonal projection on $L^{2}(M)$.
Moreover,
the operator
\begin{equation}
	N_{\lambda} \coloneqq \int_{\bbR} t^{- 1} \chi_{\clcl{- \lambda}{\lambda}^{c}}(t) \, d E(t)
	\colon L^{2}(M) \to L^{2}(M)
\end{equation}
is a bounded self-adjoint operator on $L^{2}(M)$,
with range contained in $\Dom \GJMS$,
and satisfies
\begin{gather}
	P N_{\lambda} + \pi_{\lambda} = I \text{ on  } L^{2}(M), \\
	N_{\lambda} P + \pi_{\lambda} = I \text{ on  } \Dom \GJMS.
\end{gather}
The author~\cite{Takeuchi2024-preprint}*{Theorem 4.9} also showed that
$N_{\lambda} \in \Hpsido^{- 4}(M)$ and $\pi_{\lambda} \in \Hpsido^{0}(M)$.
Additionally,
$\pi_{\lambda}$ is equal to $S + \ovS$
modulo $\Hpsido^{- 1}(M)$~\cite{Takeuchi2024-preprint}*{Proposition 4.6},
where $S$ is an approximate \Szego projection
constructed by Beals and Greiner~\cite{Beals-Greiner1988}*{Section 25}.

In order to prove \cref{thm:bounded-below-in-dimension-three},
we introduce a formally self-adjoint Heisenberg pseudodifferential operator $\wtP$ of order $4$ given by
\begin{equation}
	\wtP
	= \GJMS + \pi_{\lambda} (\Delta_{b} + 1)^{2} \pi_{\lambda}.
\end{equation}

\begin{lemma}
\label{lem:parametrix-of-modified-Paneitz-op}
	The operator $\wtP$ has a parametrix.
\end{lemma}

\begin{proof}
	From \cref{prop:equivalent-conditions-for-existence-of-parametrix},
	it is sufficient to find $\wtG_{1} \in \Hpsido^{- 4}(M)$
	such that
	\begin{equation}
		\wtG_{1} \wtP - I, \wtP \wtG_{1} - I \in \Hpsido^{- 1}(M).
	\end{equation}
	Define $\wtG_{1}$ by
	\begin{equation}
		\wtG_{1}
		\coloneqq N_{\lambda} + \pi_{\lambda} (\Delta_{b} + 1)^{- 2} \pi_{\lambda}.
	\end{equation}
	Then
	\begin{equation}
		\wtG_{1} \wtP
		= I - \pi_{\lambda}
		+ \pi_{\lambda} (\Delta_{b} + 1)^{- 2} \pi_{\lambda} (\Delta_{b} + 1)^{2} \pi_{\lambda}
		+ \pi_{\lambda} (\Delta_{b} + 1)^{- 2} \pi_{\lambda} \GJMS.
	\end{equation}
	We derive from $\pi_{\lambda} \equiv S + \ovS$ modulo $\Hpsido^{- 1}(M)$
	and \cite{Takeuchi2024-preprint}*{Lemma 4.2 and Theorem 4.8} that
	\begin{equation}
		\comm{\pi_{\lambda}}{(\Delta_{b} + 1)^{2}} \in \Hpsido^{3}(M),
		\qquad
		\pi_{\lambda} \GJMS = \GJMS \pi_{\lambda} \in \Hpsido^{- \infty}(M).
	\end{equation}
	Hence
	\begin{equation}
		\wtG_{1} \wtP - I
		= \pi_{\lambda} (\Delta_{b} + 1)^{- 2} \comm{\pi_{\lambda}}{(\Delta_{b} + 1)^{2}} \pi_{\lambda}
			+ \pi_{\lambda} (\Delta_{b} + 1)^{- 2} \pi_{\lambda} \GJMS
		\in \Hpsido^{- 1}(M).
	\end{equation}
	Taking the adjoint,
	we also have $\wtP \wtG_{1} - I \in \Hpsido^{- 1}(M)$.
\end{proof}

\begin{proof}[Proof of \cref{thm:bounded-below-in-dimension-three}]
	Take a local frame $Z_{1}$ of $T^{1, 0} M$ on an open set $U \subset M$
	such that $\calL_{\theta}(Z_{1}, Z_{1}) = 2$.
	This gives an isomorphism
	\begin{equation}
		\frakg^{\ast} M |_{U} \to U \times \frakg^{\ast}.
	\end{equation}
	Under this identification,
	it is known that
	\begin{equation}
		\sigma_{2}(L_{\mu})
		= \sigma_{2}^{0}(L_{\mu}^{0}),
		\qquad
		\sigma_{0}(S)
		= \sigma_{0}^{0}(S^{0});
	\end{equation}
	see~\cite{Ponge2008}*{Section 5} and \cite{Beals-Greiner1988}*{Section 25}.
	Thus
	\begin{align}
		\sigma_{4}(\wtP)
		&= \sigma_{2}(L_{- 1}) \ast \sigma_{2}(L_{1})
			+ \sigma_{0}(\pi_{\lambda}) \ast \sigma_{2}(\Delta_{b})^{\ast 2} \ast \sigma_{0}(\pi_{\lambda}) \\
		&= \sigma_{2}^{0}(L_{- 1}^{0}) \ast^{0} \sigma_{2}^{0}(L_{1}^{0})
			+ \sigma_{0}^{0}(\Pi^{0}) \ast^{0} \sigma_{2}^{0}(\Delta_{b}^{0})^{\ast^{0} 2} \ast^{0} \sigma_{0}^{0}(\Pi^{0}) \\
		&= \sigma_{4}^{0}(\wtP^{0}).
	\end{align}
	
	Let $q^{0}$ be the symbol furnished by \cref{thm:square-root-of-critical-CR-GJMS-operator} with $n = 1$.
	Under the above local identifications,
	$q^{0}$ defines local symbols that agree on overlaps,
	since the transition maps are $U(1)$-valued.
	They therefore determine a global real-valued symbol $q \in \Hsymbol^{2}(M)$ satisfying
	\begin{equation}
		\sigma_{4}(\wtP) = q * q.
	\end{equation}
	
	Take $u \in \Dom \GJMS$.
	It follows from $N_{\lambda} \in \Hpsido^{- 4}(M)$ that
	$(I - \pi_{\lambda}) u = N_{\lambda} \GJMS u \in \HSobolev^{4}(M)$.
	Moreover,
	we have $\wtP (I - \pi_{\lambda}) u = \GJMS (I - \pi_{\lambda}) u$.
	By \cref{lem:equivalent-conditions-for-boundedness,lem:parametrix-of-modified-Paneitz-op},
	we have
	\begin{equation}
		(\GJMS (I - \pi_{\lambda}) u, (I - \pi_{\lambda}) u)_{0}
		= (\wtP (I - \pi_{\lambda}) u, (I - \pi_{\lambda}) u)_{0}
		\geq - C \norm{(I - \pi_{\lambda}) u}_{0}^{2}.
	\end{equation}
	On the other hand,
	the definition of $\pi_{\lambda}$ implies
	\begin{equation}
		(\GJMS \pi_{\lambda} u, \pi_{\lambda} u)_{0} \geq - \lambda \norm{\pi_{\lambda} u}_{0}^{2}.
	\end{equation}
	Thus we have
	\begin{align}
		(\GJMS u, u)_{0}
		&= (\GJMS (I - \pi_{\lambda}) u, (I - \pi_{\lambda}) u)_{0}
			+ (\GJMS \pi_{\lambda} u, \pi_{\lambda} u)_{0} \\
		&\geq - C \norm{(I - \pi_{\lambda}) u}_{0}^{2} - \lambda \norm{\pi_{\lambda} u}_{0}^{2} \\
		&\geq - \max\{C, \lambda\} \norm{u}_{0}^{2},
	\end{align}
	which means that $\GJMS$ is bounded below.
\end{proof}

As an application,
we consider the CR Paneitz operator on the Rossi sphere.
The unit sphere
\begin{equation}
	S^{3}
	\coloneqq \Set{(z, w) \in \bbC^{2} | \abs{z}^{2} + \abs{w}^{2} = 1}
\end{equation}
has the canonical CR structure $T^{1, 0} S^{3}$.
This CR structure is spanned by
\begin{equation}
	Z_{1}
	\coloneqq \ovw \frac{\del}{\del z} - \ovz \frac{\del}{\del w}.
\end{equation}
A canonical contact form $\theta$ on $S^{3}$ is given by
\begin{equation}
	\theta
	\coloneqq \frac{\sqrt{- 1}}{2} (z d \ovz + w d \ovw - \ovz d z - \ovw d w)|_{S^{3}}.
\end{equation}
For a real number $t$ satisfying $0 < \abs{t} < 1$,
the \emph{Rossi sphere} $(S^{3}_{t}, T^{1, 0} S^{3}_{t})$ is defined by
\begin{equation}
	(S^{3}_{t}, T^{1, 0} S^{3}_{t})
	\coloneqq (S^{3}, \bbC (Z_{1} + t Z_{\overline{1}})).
\end{equation}
Denote by $\GJMS(t)$ the CR Paneitz operator with respect to $(S^{3}_{t}, T^{1,0}S^{3}_{t}, \theta)$.
The author~\cite{Takeuchi2024-preprint}*{Theorem 1.3} proved that
$\GJMS(t)$ has infinitely many negative eigenvalues.
As an application of \cref{thm:bounded-below-in-dimension-three},
we determine the accumulation behavior of these negative eigenvalues.

\begin{theorem}
\label{thm:negative-eigenvalues-on-Rossi-sphere}
	The negative spectrum of $\GJMS(t)$ consists of eigenvalues of finite multiplicity
	and has $0$ as its only accumulation point.
	In particular,
	$\GJMS(t)$ does not have closed range in $L^{2}(S^{3})$.
\end{theorem}

\begin{proof}
	By \cite{Takeuchi2024-preprint}*{Theorems 1.2 and 1.3},
	$\GJMS(t)$ has infinitely many distinct negative eigenvalues,
	each of finite multiplicity,
	and these eigenvalues have no accumulation point in $\opop{- \infty}{0}$.
	Since the spectrum of $\GJMS(t)$ is bounded below by \cref{thm:bounded-below-in-dimension-three},
	its negative spectrum must have $0$ as its only accumulation point.
	Choose negative eigenvalues $\lambda_{j} \to 0$
	and corresponding eigenfunctions $u_{j}$ with $\norm{u_{j}}_{0} = 1$.
	Since $\GJMS(t)$ is self-adjoint and $\lambda_{j} \neq 0$,
	we have $u_{j} \perp \Ker \GJMS(t)$.
	If $\GJMS(t)$ had closed range,
	there would exist a constant $c > 0$ such that
	\begin{equation}
		\norm{\GJMS(t) v}_{0} \geq c \norm{v}_{0}
	\end{equation}
	for every $v \in \Dom \GJMS(t) \cap (\Ker \GJMS(t))^{\perp}$.
	This contradicts
	\begin{equation}
		\norm{\GJMS(t) u_{j}}_{0} = \abs{\lambda_{j}} \longrightarrow 0.
	\end{equation}
	Thus $\GJMS(t)$ does not have closed range.
\end{proof}

\begin{remark}
\label{rem:infinitely-many-negative-eigenvalues}
	More generally,
	let $\GJMS$ be the CR Paneitz operator on a closed three-dimensional pseudo-Hermitian manifold.
	If $\GJMS$ has infinitely many negative eigenvalues,
	then \cref{thm:bounded-below-in-dimension-three}
	and \cite{Takeuchi2024-preprint}*{Theorem 1.2} imply that
	zero is their only accumulation point.
	The same argument shows that $\GJMS$ does not have closed range.
	In particular,
	this conclusion also applies to the non-embeddable three-dimensional tori
	with infinitely many negative eigenvalues constructed in \cite{Ho-Takeuchi2026-preprint}.
\end{remark}

\section*{Declaration of generative AI use}

The author used ChatGPT (OpenAI) to assist with English-language editing,
drafting and revising the abstract and introduction,
identifying relevant references,
and reviewing and refining mathematical arguments.
The author critically evaluated and revised all AI-assisted material incorporated into the manuscript,
independently verified the mathematical statements, proofs, and cited sources,
and takes full responsibility for the accuracy and integrity of the final manuscript.

\bibliography{my-reference,my-reference-preprint}

\end{document}